\documentclass[reqno, 11pt]{amsart}
\usepackage[text={155mm,235mm},centering]{geometry}
\input diagxy
\input xy

\usepackage[active]{srcltx} 

\newtheorem{thm}{Theorem}[section]

\newtheorem{lem}[thm]{Lemma}
\newtheorem{prop}[thm]{Proposition}

\theoremstyle{definition}
\newtheorem{defn}[thm]{Definition}
\theoremstyle{property}

\theoremstyle{remark}

\newtheorem{ex}[thm]{Example}
\numberwithin{equation}{section}
\newtheorem*{thm*}{Theorem}

\usepackage[mathscr]{eucal}
\usepackage[all]{xy}
\usepackage{mathrsfs}
\usepackage{xypic}
\usepackage{amsfonts}
\usepackage{amsmath}
\usepackage{amsthm,bm}
\usepackage{amssymb,tikz-cd}
\usepackage{latexsym}
\usepackage{tabularx}
\usepackage{graphicx}
\usepackage{pict2e}
\usepackage{hyperref}
\usepackage{tikz}
\usepackage{textcomp}
\usepackage[scr=rsfs]{mathalfa}
\usepackage[all]{xy}
\usepackage{tikz-cd}
\definecolor{ceruleanblue}{rgb}{0.16, 0.32, 0.75}
\hypersetup{
  colorlinks=true,
  allcolors=ceruleanblue,
  pdfauthor={},
  pdftitle={De Rham Cohomology of Quotients by Locally Free Lie Group Actions}
}
\allowdisplaybreaks

\begin{document}

\title[Descent of Basic Forms to Quotients]
{Descent of Basic Forms to Quotients by Locally Free Lie Group Actions}

\author{Yi Lin}

\address{Yi Lin \\ Department of Mathematical Sciences \\  Georgia Southern University \\Statesboro, GA, 30460 USA}
\email{yilin@georgiasouthern.edu}




\begin{abstract}
We prove an equivariant version of the theorem of Hector,
Mac\'{\i}as-Virg\'os, and Sanmart\'{\i}n-Carb\'on identifying
diffeological forms on the leaf space of a foliation with basic forms.
For any group $K$ acting by foliation-preserving diffeomorphisms, we show
that this identification is an isomorphism of $K$-equivariant cochain
complexes.

We then establish a descent theorem for basic differential forms under
locally free Lie group actions without assuming properness. Let a Lie
group $H$, not necessarily connected or second countable, act smoothly
and locally freely on a second countable manifold $M$, and let
$\mathcal F$ be the foliation by $H_0$-orbits. We prove that pullback
induces an isomorphism
\[
\Omega^\bullet(M/H)\cong\Omega^\bullet(M,\mathcal F)^H
\]
provided that $H$ is second countable or that the induced action of
$H/H_0$ on $M/H_0$ satisfies a natural subduction condition. We also give
a smooth free action for which descent fails, showing that an additional
hypothesis is genuinely necessary.

\end{abstract}

\maketitle
\maketitle
\maketitle
\section{Introduction}

Let $(M,\mathcal F)$ be a foliated manifold, and let
$\pi:M\to M/\mathcal F$ be the quotient map, where $M/\mathcal F$ is
equipped with the quotient diffeology. Hector, Mac\'{\i}as-Virg\'os, and
Sanmart\'{\i}n-Carb\'on proved~\cite{HMS} that the diffeological de Rham
complex of $M/\mathcal F$ is naturally isomorphic to the basic de Rham
complex of $(M,\mathcal F)$. We refer to this result as the HMS theorem.
Thus the diffeological de Rham cohomology of the leaf space, which is
usually difficult to compute directly, can be computed using the basic
cohomology of the foliation. More recently, Miyamoto~\cite{M22} gave a Lie
groupoid-theoretic proof of this result and extended the theorem to
important classes of singular foliations, including singular Riemannian
foliations.

This paper has two main results. The first is an equivariant version of
the HMS theorem. Suppose that a group $K$ acts on $(M,\mathcal F)$ by
foliation-preserving diffeomorphisms. Then the action descends to the leaf
space $M/\mathcal F$, and we prove that the canonical pullback
\[
\pi^*:\Omega^\bullet(M/\mathcal F)
\longrightarrow \Omega^\bullet(M,\mathcal F)
\]
is a $K$-equivariant isomorphism of cochain complexes. This equivariant
formulation broadens the applicability of the HMS theorem and is a key
ingredient in the proof of our second main result. The original proof of
Hector--Mac\'{\i}as-Virg\'os--Sanmart\'{\i}n-Carb\'on proceeds via the
holonomy pseudogroup. Our proof instead uses Haefliger's cocycle
description of a foliation. This allows us to track the group action
explicitly, although the surjectivity argument invokes a technical lemma
from~\cite[Lemma~3.3]{HMS}.

The principal application of this equivariant identification---and the
subject of our second main result---is the descent of basic differential
forms to quotients by Lie group actions. Let a Lie group $H$ act smoothly
on a manifold $M$. A differential form on $M$ is called \emph{$H$-basic}
if it is $H$-invariant and horizontal with respect to the infinitesimal
$H$-action. Pullback by the quotient map $\pi_H:M\to M/H$ sends every
diffeological differential form on $M/H$ to an $H$-basic form on $M$. We
ask when the converse holds: under what conditions does an $H$-basic form
on $M$ descend to a diffeological differential form on $M/H$?

Historically, the diffeological study of this descent question began with
Watts's work on compact Lie group actions~\cite{W12}. Karshon and
Watts~\cite{KW} subsequently established a more general result for Lie
group actions. They proved that pullback from $M/H$ is always injective
and that its image is contained in the complex of $H$-basic forms. They
further proved that every $H$-basic form descends to $M/H$ if the identity
component $H_0$ acts properly, without requiring the $H$-action itself to
be free or proper. Watts also extended this identification to proper Lie
groupoids and obtained an intrinsic de Rham theorem for their orbit
spaces~\cite{W22}.

We address the descent question for locally free Lie group actions without
any properness assumption. Let $H$ be a Lie group, not necessarily connected
or second countable, and let $M$ be a second countable manifold equipped
with a smooth, locally free right $H$-action. Let $H_0$ denote the identity
component of $H$, and let $\mathcal F$ be the foliation of $M$ by
$H_0$-orbits. Since $H$ and $H_0$ have the same Lie algebra, the $H$-basic
forms are precisely the $H$-invariant forms that are basic for $\mathcal F$.
The resulting complex is therefore $\Omega^\bullet(M,\mathcal F)^H$.

We prove that every $H$-basic form descends uniquely to $M/H$ under either
of the following two natural hypotheses: (i) $H$ is second countable; or
(ii) the induced action of the component group $H/H_0$ on $M/H_0$ satisfies
a natural subduction condition. Equivalently, pullback induces an
isomorphism of cochain complexes
\[
\pi_H^*:\Omega^\bullet(M/H)
\xrightarrow{\;\cong\;}
\Omega^\bullet(M,\mathcal F)^H.
\]
Our result thus complements the theorem of Karshon and Watts. Their theorem
permits actions that are not locally free but assumes properness of the
identity-component action, whereas ours treats locally free actions without
assuming properness.

We apply the equivariant HMS theorem after first passing to the quotient by
$H_0$. It identifies differential forms on $M/H_0$ with forms on $M$ that
are basic for $\mathcal F$, while equivariance accounts for the residual
action of $H/H_0$. The remaining issue is whether an $H/H_0$-invariant form
on $M/H_0$ descends through
\[
M/H_0\longrightarrow (M/H_0)/(H/H_0)=M/H.
\]
The subduction condition provides one mechanism for establishing this
descent, while second countability of $H$ provides another via an extension
of the Haefliger-cocycle argument that incorporates the residual action of
$H/H_0$.

The need for an additional hypothesis is genuine. In
Example~\ref{ex:subduction-fails}, we give a smooth free action for which
$M/H$ is a single point, while the complex of $H$-basic forms contains a
nonzero $1$-form. We also show that second countability of $H$ does not
itself imply the subduction condition. Thus the two hypotheses provide
distinct sufficient conditions for the desired isomorphism.

Our theorem identifies $H^\bullet_{\mathrm{dR}}(M/H)$ with the cohomology
of the complex of $H$-basic forms on $M$. This complex is typically more
accessible than the diffeological de Rham complex of the quotient. This
computational perspective is complementary to the approach developed in
our recent work~\cite{L26}. In that work, we draw on
substantial machinery from Riemannian foliation theory to construct a
transverse averaging operator on the basic de Rham complex in the presence
of an isometric transverse Lie algebra action. The resulting operator
yields invariant representatives of basic cohomology classes, thereby
providing a useful tool for computing basic cohomology. For the problem
considered here, however, the method is more direct: it studies when basic
differential forms descend to diffeological quotients.

An important homogeneous application concerns a recent result of Clark and
Ziegler~\cite{CZ26}. Let $H$ be a dense subgroup of a Lie group $G$, and
write $\mathfrak g$ and $\mathfrak h$ for the Lie algebras of $G$ and $H$,
respectively. Clark and Ziegler proved that $\mathfrak h$ is an ideal in
$\mathfrak g$ and identified the diffeological de Rham complex of $G/H$
with the Chevalley--Eilenberg complex of $\mathfrak g/\mathfrak h$.
Consequently,
\[
H^\bullet_{\mathrm{dR}}(G/H)
\cong H^\bullet(\mathfrak g/\mathfrak h).
\]
Their argument uses density in an essential way to replace $H$-invariance
by $G$-invariance at the cochain level.

In the homogeneous case $M=G$, with $H$ acting by right translations, our
theorem recovers the Clark--Ziegler dense-subgroup theorem through a
foliation-theoretic argument. More generally, our descent method does not
require $H$ to be dense in $G$. This broader applicability is illustrated
in Example~\ref{ex:non-second-countable-nondense}, where $H$ is nonclosed,
its closure is a proper subgroup of $G$, and its component group is
non-second-countable. In this case, our method completely determines the
diffeological de Rham cohomology of $G/H$.

Throughout the paper, all manifolds are assumed to be second countable,
except for Lie groups that are explicitly allowed to be
non-second-countable. The paper is organized as follows.
Section~\ref{diffeology} recalls the necessary diffeological background
and proves the equivariant HMS theorem, Theorem~\ref{Hector}.
Section~\ref{sec:CZthm} establishes the descent theorem,
Theorem~\ref{M/H}, examines the two sufficient conditions for descent,
and derives the dense-subgroup theorem of Clark and
Ziegler~\cite{CZ26} through a foliation-theoretic argument.

\section{De Rham cohomology of a diffeological space}\label{diffeology}

This section fixes the diffeological conventions used throughout the paper.
We recall plots, subspace, product, and quotient diffeologies, subductions,
diffeological forms, and the associated de Rham complex. We then prove
Theorem~\ref{Hector}, an equivariant version of \cite[Theorem 3.5]{HMS},
which is a key ingredient in the proof of Theorem~\ref{M/H}. For general
background on diffeology, we refer the reader to~\cite{I13}.
\begin{defn}
Let $X$ be a set. A \emph{parametrization} of $X$ is a map
$\alpha:U\to X$ defined on an open subset $U$ of some Euclidean space
$\mathbb R^n$, $n\geq 0$. A \emph{diffeology} $\mathcal D$ on $X$ is a family
of parametrizations satisfying the following axioms:
\begin{itemize}
\item[(1)] every constant parametrization belongs to $\mathcal D$;
\item[(2)] if $\alpha:U\to X$ belongs to $\mathcal D$ and $h:V\to U$ is a
smooth map between open subsets of Euclidean spaces, then
$\alpha\circ h$ belongs to $\mathcal D$;
\item[(3)] if $\alpha:U\to X$ is a parametrization such that every point
$p\in U$ has an open neighborhood $V_p\subset U$ with
$\alpha|_{V_p}\in\mathcal D$, then $\alpha\in\mathcal D$.
\end{itemize}
A parametrization belonging to a diffeology $\mathcal D$ is called a \emph{plot}. The pair
$(X,\mathcal D)$ is called a \emph{diffeological space}.

Let $(X,\mathcal D)$ be a diffeological space. A family of plots
$\mathcal C$ is called a \emph{generating family} for $\mathcal D$ if every
plot $f:U\to X$ is either constant, or, for every $x\in U$, there exist an open
neighborhood $V_x\subset U$, a plot $g:W\to X$ in $\mathcal C$, and a smooth map
$h:V_x\to W$ such that $f|_{V_x}=g\circ h$.
\end{defn}

\begin{defn}
A map $f:(X,\mathcal D)\to (Y,\mathcal D')$ between diffeological spaces is said
to be \emph{smooth} if, for every plot $\alpha\in\mathcal D$, the composite
$f\circ \alpha$ belongs to $\mathcal D'$. A smooth map $f$ is called a
\emph{diffeological diffeomorphism} if it is bijective and its inverse
$f^{-1}:(Y,\mathcal D')\to (X,\mathcal D)$ is also smooth.
\end{defn}

\begin{defn}\label{subspace}
Let $(X,\mathcal D)$ be a diffeological space and let $A\subset X$. The
\emph{subspace diffeology} on $A$ is the diffeology whose plots are precisely
those parametrizations $\alpha:U\to A$ such that the composite
$U\xrightarrow{\alpha} A\hookrightarrow X$ is a plot of $X$.
\end{defn}
\begin{defn}\label{product}
Let $(X,\mathcal D_X)$ and $(Y,\mathcal D_Y)$ be diffeological spaces. The
\emph{product diffeology} on $X\times Y$ is the diffeology whose plots are
precisely the parametrizations $\alpha:U\to X\times Y$ such that both
component maps $\operatorname{pr}_X\circ\alpha:U\to X$ and
$\operatorname{pr}_Y\circ\alpha:U\to Y$ are plots.
\end{defn}

\begin{defn}\label{forms}
Let $(X,\mathcal D)$ be a diffeological space. A differential $k$-form
$\gamma$ on $(X,\mathcal D)$ assigns to each plot $\alpha:U\to X$ a
differential $k$-form $\gamma_\alpha$ on $U$, subject to the compatibility
condition
\begin{equation}\label{compatibility}
\gamma_{\alpha\circ f}=f^*\gamma_\alpha
\end{equation}
for every smooth map $f:V\to U$ between open subsets of Euclidean spaces. We
denote the space of diffeological forms by $\Omega^\bullet(X,\mathcal D)$.
The exterior derivative is defined plotwise by
$d\{\gamma_\alpha\}:=\{d\gamma_\alpha\}$.
\end{defn}

Then $d^2=0$, and the \emph{de Rham cohomology} of $(X,\mathcal D)$ is
\[
H^\bullet_{\mathrm{dR}}(X,\mathcal D):=\ker(d)/\operatorname{im}(d).
\]
When the diffeology is clear, we write simply $H^\bullet_{\mathrm{dR}}(X)$.

If $f:(X,\mathcal D)\to (Y,\mathcal D')$ is smooth, its pullback is given by
\[
f^*:\Omega^\bullet(Y,\mathcal D')\to \Omega^\bullet(X,\mathcal D),
\qquad
(f^*\gamma)_\alpha:=\gamma_{f\circ\alpha}.
\]
It commutes with $d$ and hence induces
$f^*:H^\bullet_{\mathrm{dR}}(Y,\mathcal D')\to
H^\bullet_{\mathrm{dR}}(X,\mathcal D)$.

\begin{ex}[Manifold diffeology]
Let $\{(V_i,\phi_i)\}_{i\in I}$ be an atlas on a differentiable manifold
$M$. The finest diffeology on $M$ containing
$\mathcal C=\{\phi_i^{-1}\}_{i\in I}$ is called the \emph{manifold
diffeology} of $M$, and $\mathcal C$ is a generating family for it. When
$M$ is equipped with this diffeology, the diffeological de Rham complex
$\Omega^\bullet(M)$ agrees with the usual de Rham complex of $M$.
\end{ex}

\begin{defn}\label{quotient}
Let $(X,\mathcal{D})$ be a diffeological space, let $\sim$ be an equivalence
relation on $X$, and let $\pi:X\to X/\sim$ be the quotient map. Define
$\mathcal{D}'$ to be the family of parametrizations $\gamma:U\to X/\sim$
such that either $\gamma$ is constant, or for every $x\in U$ there exist
a neighborhood $V$ of $x$ and a plot $\alpha\in \mathcal{D}$ defined on
$V$ such that
\begin{equation}\label{quotient-plot}
\gamma|_V=\pi\circ \alpha.
\end{equation}

Then $\mathcal{D}'$ is a diffeology on $X/\sim$, called the
\emph{quotient diffeology}. It is the finest diffeology on $X/\sim$ for
which $\pi$ is smooth. A smooth map $p:X\to Y$ between diffeological
spaces is called a \emph{subduction} if it is surjective and $Y$ carries
the quotient diffeology induced by $p$.
\end{defn}

By~\cite[Article~1.48]{I13}, a smooth surjection $p:X\to Y$ is a
subduction if and only if for every plot $f:U\to Y$ and every $x\in U$,
there exist an open neighborhood $V\subset U$ of $x$ and a plot
$g:V\to X$ such that $f|_V=p\circ g$.

\begin{defn}
A right action of a group $H$ on a diffeological space $(X,\mathcal D)$ is
an action on the underlying set such that, for every $h\in H$, the map
$R_h:X\to X$, $x\mapsto x\cdot h$, is a diffeological diffeomorphism.
\end{defn}

Such an action induces pullback operators on
$\Omega^\bullet(X,\mathcal D)$. For $h\in H$, $\alpha\in
\Omega^\bullet(X,\mathcal D)$, and a plot $f:U\to X$, define
$(h^*\alpha)_f:=\alpha_{R_h\circ f}$. Each $h^*$ is a cochain automorphism of
$\Omega^\bullet(X,\mathcal D)$. We write $\Omega^\bullet(X,\mathcal D)^H$ for 
the subcomplex of forms fixed
by $h^*$ for every $h\in H$.

\begin{lem}\label{injective-pullback}
Let $p:(X,\mathcal D_X)\to(Y,\mathcal D_Y)$ be a subduction of
diffeological spaces. Then the pullback map
$p^*:\Omega^\bullet(Y,\mathcal D_Y)\to
\Omega^\bullet(X,\mathcal D_X)$ is injective.

In particular, suppose that a group $H$ acts on a diffeological space
$(X,\mathcal D)$, and equip $X/H$ with the quotient diffeology
$\mathcal D_H$. Then the quotient map $\pi_H:X\to X/H$ induces an injective
morphism of de Rham complexes
$\pi_H^*:\Omega^\bullet(X/H,\mathcal D_H)\to
\Omega^\bullet(X,\mathcal D)^H$.
\end{lem}

\begin{proof}
Let $\alpha\in\Omega^\bullet(Y,\mathcal D_Y)$ satisfy $p^*\alpha=0$, and let
$f:U\to Y$ be a plot. Since $p$ is a subduction, for every $u\in U$ there
exist an open neighborhood $V_u\subset U$ and a plot $g_u:V_u\to X$ such
that $f|_{V_u}=p\circ g_u$. By the compatibility condition for
diffeological forms,
$\alpha_f|_{V_u}=\alpha_{f|_{V_u}}
=\alpha_{p\circ g_u}=(p^*\alpha)_{g_u}=0$.
Since the sets $V_u$ cover $U$, it follows that $\alpha_f=0$. As $f$ was
arbitrary, $\alpha=0$. Hence $p^*$ is injective.

For the special case, $\pi_H$ is a subduction by the definition of the
quotient diffeology. Moreover, $\pi_H\circ R_h=\pi_H$ for every $h\in H$.
Thus, for every $\alpha\in\Omega^\bullet(X/H,\mathcal D_H)$, one has
$R_h^*(\pi_H^*\alpha)=\pi_H^*\alpha$. Therefore $\pi_H^*\alpha$ is
$H$-invariant, and the result follows from the first part.
\end{proof}

We now recall the Haefliger cocycle description of a foliation, which is the
framework used in the proof of Theorem~\ref{Hector}. We also fix the notation
for basic forms and basic cohomology.

\begin{defn}\label{foliation}
A codimension $q$ \emph{foliation} $\mathcal{F}$ on $M$ is given by a
maximal family of submersions $\{f_i\colon U_i\to\mathbb R^q\}_{i\in I}$,
where $\{U_i\}_{i\in I}$ is an open cover of $M$, such that for every
$i,j\in I$ and every $x\in U_i\cap U_j$ there is a germ of diffeomorphism
$\phi_{ji}^x$ of $\mathbb R^q$ at $f_i(x)$, with $f_j=\phi_{ji}^x\circ f_i$
on a neighborhood $V_x\subset U_i\cap U_j$ of $x$. The germs $\phi_{ji}^x$
satisfy the $1$-cocycle condition
$\phi_{ki}^x=\phi_{kj}^x\circ\phi_{ji}^x$ at every $x\in U_i\cap U_j\cap U_k$.
\end{defn}

The Haefliger cocycle describes the transverse structure of the foliation.
Equivalently, it determines the integrable tangent distribution
$T\mathcal F\subset TM$, and the next definition uses this tangent
description.

\begin{defn}\label{basic-forms}
Let $(M,\mathcal F)$ be a foliated manifold. A differential form
$\omega\in\Omega^\bullet(M)$ is called \emph{basic} if
\[
\iota_X\omega=0
\qquad\text{and}\qquad
\mathcal L_X\omega=0
\]
for every vector field $X$ tangent to $\mathcal F$. We denote the complex of
basic forms by $\Omega^\bullet(M,\mathcal F)$. Its cohomology is called the
\emph{basic cohomology} of $(M,\mathcal F)$ and is denoted by
$H^\bullet_{\mathrm{dR}}(M,\mathcal F)$.
\end{defn}

We shall use the following elementary consequence of the Baire category
argument appearing in \cite[Lemma~3.3]{HMS}. We state it in the language of
Haefliger cocycles, since this is the form needed below. If
$\{(U_a,f_a)\}_{a\in A}$ is a countable Haefliger cocycle for a codimension
$q$ foliation, we write
$Q:=\coprod_{a\in A} f_a(U_a)$ for the disjoint union of its transverse
coordinate domains.

\begin{lem}[Reformulation of {\cite[Lemma~3.3]{HMS}}]\label{HMS-lemma}
Let $\{(U_a,f_a)\}_{a\in A}$ be a countable Haefliger cocycle for a
codimension $q$ foliation, and set $Q:=\coprod_{a\in A}f_a(U_a)$.
Let $\mathscr C$ be a countable family of local diffeomorphisms between open
subsets of $Q$, and let $\sim_{\mathscr C}$ be the equivalence relation on
$Q$ generated by $\mathscr C$. Assume that whenever
$x\sim_{\mathscr C}y$, there exists $\theta\in\mathscr C$, defined near
$x$, such that $\theta(x)=y$. Let
$\pi_{\mathscr C}:Q\to Q/\!\sim_{\mathscr C}$ be the quotient map.

If $\alpha,\beta:U\to Q$ are smooth maps from an open set
$U\subset\mathbb R^n$ such that
$\pi_{\mathscr C}\circ\alpha=\pi_{\mathscr C}\circ\beta$, then there is a
countable family of open subsets $W_s\subset U$ with dense union and
elements $\theta_s\in\mathscr C$ such that
$\beta=\theta_s\circ\alpha$ on $W_s$.
\end{lem}

We now prove the equivariant version of the
HMS theorem.

\begin{thm}\label{Hector}
Let $(M,\mathcal F)$ be a foliated manifold, and let $H$ act on the right
on $(M,\mathcal F)$ by foliation automorphisms. Then the action of $H$
naturally descends to an action on the leaf space $M/\mathcal{F}$, with each
$h\in H$ acting as a diffeological diffeomorphism with respect to the quotient
diffeology $\mathcal{D}'$. Moreover, the quotient map $\pi:M\to M/\mathcal{F}$
induces a canonical $H$-equivariant isomorphism of de Rham complexes
\[
\pi^* : (\Omega^\bullet(M/\mathcal{F},\mathcal{D}'),d)
\xrightarrow{\;\cong\;}
(\Omega^\bullet(M,\mathcal{F}),d),
\]
and in particular an $H$-equivariant isomorphism
$H^\bullet_{\mathrm{dR}}(M/\mathcal{F},\mathcal{D}')\cong H^\bullet_{\mathrm{dR}}(M,\mathcal{F})$.
\end{thm}

\begin{proof}
Since each $h\in H$ maps leaves of $\mathcal F$ to leaves, the action of
$H$ on $M$ descends naturally to an action on $M/\mathcal F$. Moreover,
each induced map on $M/\mathcal F$ is a diffeological diffeomorphism with
respect to the quotient diffeology $\mathcal D'$. We now prove the remaining
assertions in several steps.

\begin{itemize}

\item[i)] Since $\pi:M\to M/\mathcal F$ is a subduction,
Lemma~\ref{injective-pullback} implies that pullback by $\pi$ is injective.
It remains to show that its image lies in $\Omega^\bullet(M,\mathcal F)$.

Since $M$ is second countable, we may choose a countable refinement
$\{f_i:U_i\to f_i(U_i)\subset\mathbb R^q\}_{i\in J}$ of the Haefliger
cocycle ... such that $\{U_i\}_{i\in J}$ covers $M$ and each $f_i$ admits a
section $s_i:f_i(U_i)\to U_i$ for which $x$ and $s_i(f_i(x))$ lie on the
same leaf for every $x\in U_i$. Set
$q_i:=\pi\circ s_i:f_i(U_i)\to M/\mathcal F$. Then
$q_i\circ f_i=\pi|_{U_i}$.

Set $\mathcal C:=\{q_i:f_i(U_i)\to M/\mathcal F\}_{i\in J}$. Then
$\mathcal C$ is a generating family for the quotient diffeology
$\mathcal D'$ on $M/\mathcal F$. Indeed, if $f:O\to M/\mathcal F$ is a plot
and $x\in O$, then, after shrinking to a neighborhood $V$ of $x$, there is a
smooth lift $\widetilde f:V\to M$ such that
$f|_V=\pi\circ\widetilde f$. After shrinking $V$ further, we may assume that
$\widetilde f(V)\subset U_i$ for some $i\in J$. Then
$f|_V=\pi\circ\widetilde f=q_i\circ f_i\circ\widetilde f$. Thus every plot
of $M/\mathcal F$ locally factors through one of the maps $q_i$.

Hence, for $\alpha\in\Omega^\bullet(M/\mathcal F,\mathcal D')$, we have
$(\pi^*\alpha)|_{U_i}=f_i^*\alpha_{q_i}$. Since $f_i$ is a submersion
defining the foliation on $U_i$, the form $f_i^*\alpha_{q_i}$ is basic on
$U_i$. Thus $\pi^*\alpha$ is basic on each $U_i$, and therefore
$\pi^*\alpha\in\Omega^\bullet(M,\mathcal F)$.

\item[ii)] We next show that $\pi^*$ is $H$-equivariant. Let $h\in H$, and
let $\overline{R_h}:M/\mathcal F\to M/\mathcal F$ be the induced diffeomorphism of
the leaf space, so that $\pi\circ R_h=\overline{R_h}\circ\pi$. By functoriality of
pullback, for every $\alpha\in\Omega^\bullet(M/\mathcal F,\mathcal D')$,
\[
(R_h)^*\pi^*\alpha=(\pi\circ R_h)^*\alpha
=(\overline{R_h}\circ\pi)^*\alpha=\pi^*(\overline{R_h})^*\alpha.
\]
Thus $\pi^*$ is $H$-equivariant.

\item[iii)] Finally, we show that $\pi^*$ is surjective. Let
$\gamma\in\Omega^k(M,\mathcal F)$ be a basic $k$-form. We construct a
diffeological $k$-form $\alpha$ on $M/\mathcal F$ such that
$\pi^*\alpha=\gamma$.

For each $i\in J$, since $f_i$ is a submersion and
$\gamma|_{U_i}$ is basic, there is a unique form
$\alpha_{q_i}\in\Omega^k(f_i(U_i))$ satisfying
$f_i^*\alpha_{q_i}=\gamma|_{U_i}$; equivalently,
$\alpha_{q_i}=s_i^*\gamma$.

Let $f:O\to M/\mathcal F$ be a plot in $\mathcal D'$. Since $\mathcal C$ is
a generating family, there is an open cover $O=\bigcup_\lambda O_\lambda$
such that
$f|_{O_\lambda}=q_{i_\lambda}\circ h_\lambda$ for some $i_\lambda\in J$ and
some smooth map $h_\lambda:O_\lambda\to f_{i_\lambda}(U_{i_\lambda})$.
Define $\alpha_f|_{O_\lambda}:=h_\lambda^*\alpha_{q_{i_\lambda}}$.

We claim that these local definitions agree on overlaps. Suppose that on
$V:=O_\lambda\cap O_\mu$ we have
$f|_V=q_i\circ h_i=q_j\circ h_j$, with
$h_i:V\to f_i(U_i)$ and $h_j:V\to f_j(U_j)$. Set
$Q:=\coprod_{a\in J}f_a(U_a)$. Choose a fixed countable basis of the
transverse coordinate domains, and let $\mathscr H$ be the countable family
of local holonomy transformations obtained from the Haefliger transition
maps and their inverses by taking finite compositions and restricting
domains and ranges to this basis. This family locally generates the
holonomy equivalence relation on $Q$.

Regard $h_i$ and $h_j$ as maps into $Q$. The equality
$q_i\circ h_i=q_j\circ h_j$ implies that their values are pointwise
holonomy equivalent. Hence, by Lemma~\ref{HMS-lemma}, there is a countable
family of open sets $W_s\subset V$ with dense union such that
$h_j=\phi^s\circ h_i$ on $W_s$ for some $\phi^s\in\mathscr H$.

Since $\gamma$ is basic, the transverse forms $\alpha_{q_i}$ are invariant
under holonomy. Indeed, for an elementary Haefliger transition
$\phi:D\subset f_i(U_i)\to f_j(U_j)$, on an appropriate overlap $A$ one has
$f_j|_A=\phi\circ f_i|_A$, and hence
\[
(f_i|_A)^*\phi^*\alpha_{q_j}
=(f_j|_A)^*\alpha_{q_j}
=\gamma|_A
=(f_i|_A)^*\alpha_{q_i}.
\]
Since $f_i|_A$ is a submersion, its pullback is injective, so
$\phi^*\alpha_{q_j}=\alpha_{q_i}$. The same identity holds for restrictions,
inverses, and finite compositions. Therefore, on each $W_s$,
$h_j^*\alpha_{q_j}=h_i^*(\phi^s)^*\alpha_{q_j}=h_i^*\alpha_{q_i}$.
Since $\bigcup_sW_s$ is dense in $V$, the equality extends to all of $V$.

Thus the local forms $h_\lambda^*\alpha_{q_{i_\lambda}}$ glue to a
well-defined form $\alpha_f\in\Omega^k(O)$. The construction is independent
of the chosen factorizations by the same overlap argument.

It remains to check compatibility with reparametrization. Let
$\varphi:O'\to O$ be smooth. If
$f|_{O_\lambda}=q_{i_\lambda}\circ h_\lambda$, then
$(f\circ\varphi)|_{\varphi^{-1}(O_\lambda)}
=q_{i_\lambda}\circ(h_\lambda\circ\varphi)$. Therefore
\[
\alpha_{f\circ\varphi}|_{\varphi^{-1}(O_\lambda)}
=(h_\lambda\circ\varphi)^*\alpha_{q_{i_\lambda}}
=\varphi^*(h_\lambda^*\alpha_{q_{i_\lambda}})
=\varphi^*(\alpha_f|_{O_\lambda}).
\]
Hence $\alpha_{f\circ\varphi}=\varphi^*\alpha_f$, so
$\alpha=\{\alpha_f\}$ is a diffeological $k$-form on $M/\mathcal F$.

Finally, for each $i\in J$ we have $\pi|_{U_i}=q_i\circ f_i$, and hence
$(\pi^*\alpha)|_{U_i}=f_i^*\alpha_{q_i}=\gamma|_{U_i}$. Since the sets
$U_i$, $i\in J$, cover $M$, it follows that $\pi^*\alpha=\gamma$.
Therefore $\pi^*$ is surjective.
\end{itemize}

Thus $\pi^*$ is an $H$-equivariant isomorphism of de Rham complexes. The
cohomological statement follows.
\end{proof}

\section{Cohomology of quotients by locally free Lie group actions}\label{sec:CZthm}

In this section we establish the second main result of the paper, concerning
quotients $M/H$ by possibly disconnected Lie group actions. More precisely,
fix a Lie group $H$, not necessarily second countable, and a second countable
manifold $M$ equipped with a smooth, locally free right $H$-action.

When $H$ is a Lie group in the classical sense, i.e. its underlying manifold
is second countable, the component group $H/H_0$ is countable, so the
Haefliger-cocycle argument used in Theorem~\ref{Hector} can be adapted. If
one allows Lie groups whose underlying manifolds are not second countable,
this countability argument is no longer available; in that setting we impose
a natural subduction condition on the action of $H/H_0$ on $M/H_0$.

We begin by recalling the canonical Lie group structure carried by an
arbitrary subgroup of a Lie group. This is the setting used by
Clark--Ziegler and clarifies the notion of subgroup used below.

\begin{prop}[Bourbaki {\cite[Chapter~III, \S 4.5, Proposition~9]{Bour89}}]\label{Bour}
Let $H$ be an arbitrary subgroup of a Lie group $G$. Then $H$ admits a unique
manifold structure for which:
\begin{itemize}
\item[(1)] the inclusion map $i:H\hookrightarrow G$ is an immersion;
\item[(2)] for every manifold $N$, a map $F:N\to H$ is smooth if and only if
$i\circ F:N\to G$ is smooth.
\end{itemize}
With this structure, $H$ is a Lie group with Lie algebra
\[
\mathfrak h=\{Z\in \mathfrak g:\exp(tZ)\in H \text{ for all } t\in \mathbb R\}.
\]
\end{prop}

Henceforth, whenever $H$ is a subgroup of a Lie group $G$, we regard $H$ as
endowed with the canonical Lie group structure of Proposition~\ref{Bour}. In
particular, $H_0$ denotes the identity component of $H$ with respect to this
intrinsic Lie group topology, and $H/H_0$ denotes its component group. Unless
explicitly stated otherwise, whenever $\overline H$ appears, the closure is
taken in the ambient Lie group $G$. For further background on this intrinsic
Lie group structure and the associated terminology, see also Appendix~A of
\cite{CZ26}.

Let $H$ be a Lie group, not necessarily second countable, acting smoothly
and locally freely on the right on a second countable manifold $M$. Let
$H_0$ be the identity component of $H$, and let $\mathcal F$ be the
foliation by the $H_0$-orbits. Set $X:=M/H_0$ and $K:=H/H_0$, and let
$\pi_0:M\to X$ and $\rho:X\to X/K\cong M/H$ be the quotient maps. The group
$K$ acts on $X$ on the right. The following descent lemma for the component
group action is a crucial step in the proof of Theorem~\ref{M/H}.

\begin{lem}\label{component-descent}
With the notation above, assume that one of the following two conditions
holds:
\begin{itemize}
\item[(a)] $H$ is second countable.
\item[(b)] The canonical map
\begin{equation}\label{canonical-map}
\Delta:X\times K\longrightarrow X\times_{X/K}X,\qquad
(x,k)\longmapsto (x,x\cdot k),
\end{equation}
is a subduction, where $X\times K$ has the product diffeology and
$X\times_{X/K}X$ has the subset diffeology inherited from $X\times X$.
\end{itemize}
Then pullback by $\rho$ induces an isomorphism of cochain complexes
\[
\rho^*:\Omega^\bullet(X/K)\xrightarrow{\;\cong\;}\Omega^\bullet(X)^K.
\]
\end{lem}

\begin{proof}
Let $\mathcal F$ denote the regular foliation of $M$ by the $H_0$-orbits, so
that $X=M/H_0=M/\mathcal F$. Since $\rho:X\to X/K$ is a subduction by the
definition of the quotient diffeology, $\rho^*$ is injective by
Lemma~\ref{injective-pullback}. It remains to prove surjectivity.

Let $\omega\in\Omega^\bullet(X)^K$. We construct a diffeological form
$\gamma\in\Omega^\bullet(X/K)$ such that $\rho^*\gamma=\omega$. Let
$f:U\to X/K$ be a plot. Since $\rho$ is a subduction, there exist an open
cover $U=\bigcup_iU_i$ and plots $g_i:U_i\to X$ such that
$f|_{U_i}=\rho\circ g_i$. Define
$\gamma_f|_{U_i}:=\omega_{g_i}$. We show that these local forms are
independent of the chosen lifts and agree on overlaps.

For $k\in K$, let $\bar R_k:X\to X$ denote the induced right action of $k$ on
$X=M/H_0$. If $h\in H$ represents $k$, then
$\bar R_k\circ\pi_0=\pi_0\circ R_h$.

First assume condition~(a). Since $H_0$ is open in $H$ and $H$ is second
countable, the discrete quotient $K=H/H_0$ is countable. Since $M$ is second
countable, the foliation $\mathcal F$ admits a countable Haefliger cocycle.
Choose such a cocycle $\{(U_a,f_a)\}_{a\in A}$ and refine it so that each
$f_a$ admits a section $s_a:f_a(U_a)\to U_a$. Set
$q_a:=\pi_0\circ s_a:f_a(U_a)\to X$. The family $\{q_a\}_{a\in A}$
generates the quotient diffeology on $X$.

Choose a set-theoretic section $\sigma:K\to H$ of the quotient map
$H\to K$, and write $h_k:=\sigma(k)$. Since the identity component $H_0$ is
normal in $H$, each right translation $R_{h_k}:M\to M$ maps $H_0$-orbits to
$H_0$-orbits. It therefore preserves the foliation $\mathcal F$ and descends
to $\bar R_k$ on $X$. After restricting the foliation charts if necessary,
the maps $R_{h_k}$ induce local diffeomorphisms between transverse coordinate
domains.

Let $\mathscr C_K$ be the countable family of local transformations obtained
from the ordinary Haefliger transitions and these induced transverse maps by
taking finite compositions and restricting their domains and ranges to a
fixed countable basis of the transverse coordinate domains. This family
locally generates the equivalence relation on
$\coprod_a f_a(U_a)$ induced by the map to $X/K$. Moreover, for each
$\theta:D\subset f_a(U_a)\to E\subset f_b(U_b)$ in $\mathscr C_K$, after
possibly restricting $D$, there exists $k_\theta\in K$ such that
$q_b\circ\theta=\bar R_{k_\theta}\circ q_a$ on $D$.

Let $g,g':V\to X$ be two local lifts of the same plot of $X/K$, so that
$\rho\circ g=\rho\circ g'$ on $V$. Fix $v_0\in V$. Since the maps $q_a$
generate the diffeology of $X$, after replacing $V$ by a neighborhood of
$v_0$, we may write $g=q_a\circ u$ and $g'=q_b\circ u'$ for smooth maps
$u:V\to f_a(U_a)$ and $u':V\to f_b(U_b)$. By the local-generation property
of $\mathscr C_K$, for every $v\in V$ there exists
$\theta\in\mathscr C_K$, defined near $u(v)$, such that
$u'(v)=\theta(u(v))$.

Applying Lemma~\ref{HMS-lemma} to the countable family $\mathscr C_K$, we
obtain open subsets $W_s\subset V$ with dense union and elements
$\theta_s\in\mathscr C_K$ such that $u'=\theta_s\circ u$ on $W_s$. Let
$k_s:=k_{\theta_s}$. Then
$g'|_{W_s}=\bar R_{k_s}\circ g|_{W_s}$, and the $K$-invariance of $\omega$
gives $\omega_{g'}|_{W_s}=\omega_g|_{W_s}$. Since $\bigcup_sW_s$ is dense
in $V$ and both sides are smooth forms, the equality extends to all of $V$.
Thus any two local lifts of the same plot of $X/K$ determine the same
pullback of $\omega$.

Now assume condition~(b). Let $g,g':V\to X$ be two local lifts of the same
plot of $X/K$. Then $(g,g'):V\to X\times_{X/K}X$ is a plot. Since $\Delta$
is a subduction, near every point of $V$ there exist an open set $W\subset V$
and a plot $(\widetilde g,\kappa):W\to X\times K$ lifting
$(g,g')|_W$. Comparing first components gives
$\widetilde g=g|_W$, and hence
$g'(w)=g(w)\cdot\kappa(w)$ for every $w\in W$. Since $K$ is discrete,
after shrinking $W$ we may assume that $\kappa$ is constant, say
$\kappa=k$. Thus $g'|_W=\bar R_k\circ g|_W$, and the $K$-invariance of
$\omega$ gives $\omega_{g'}|_W=\omega_g|_W$. Hence
$\omega_{g'}=\omega_g$ locally on $V$, and therefore on all of $V$.

In either case, any two local lifts of the same plot of $X/K$ give the same
pullback of $\omega$. Applying this to $g_i|_{U_i\cap U_j}$ and
$g_j|_{U_i\cap U_j}$ shows that the local forms $\omega_{g_i}$ and
$\omega_{g_j}$ agree on overlaps. They therefore glue to a well-defined form
$\gamma_f$ on $U$, independent of the chosen open cover and local lifts.

It remains to check compatibility with reparametrization. Let
$\varphi:V\to U$ be smooth. If $f|_{U_i}=\rho\circ g_i$, then
$(f\circ\varphi)|_{\varphi^{-1}(U_i)}
=\rho\circ(g_i\circ\varphi)$. Consequently,
$\gamma_{f\circ\varphi}|_{\varphi^{-1}(U_i)}
=\omega_{g_i\circ\varphi}
=\varphi^*\omega_{g_i}
=\varphi^*(\gamma_f|_{U_i})$.
Thus $\gamma_{f\circ\varphi}=\varphi^*\gamma_f$, and
$\gamma=\{\gamma_f\}$ defines a diffeological form on $X/K$.

Finally, let $g:U\to X$ be a plot. For the plot $\rho\circ g$, we may use
$g$ as a global lift, so $\gamma_{\rho\circ g}=\omega_g$. Hence
$(\rho^*\gamma)_g=\omega_g$ for every plot $g$ of $X$, and therefore
$\rho^*\gamma=\omega$. This proves surjectivity.
\end{proof}

We are ready to prove the second main result of this paper.

\begin{thm}\label{M/H}
Fix a Lie group $H$, not necessarily second countable, and a second
countable manifold $M$ equipped with a smooth, locally free right
$H$-action. Write $H_0$ for the identity component of $H$, and let
$\mathcal F$ be the foliation by the $H_0$-orbits. Set $X:=M/H_0$ and
$K:=H/H_0$. Assume that either
\begin{itemize}
\item[(a)] $H$ is second countable;
\item[(b)] the canonical map~\eqref{canonical-map} is a subduction, where
$X\times K$ has the product diffeology and $X\times_{X/K}X$ has the subset
diffeology inherited from $X\times X$.
\end{itemize}
Then pullback by $\pi_H:M\to M/H$ induces a natural isomorphism of cochain
complexes
\[
\pi_H^*:\Omega^\bullet(M/H)\xrightarrow{\;\cong\;}
\Omega^\bullet(M,\mathcal F)^H,
\]
and consequently a natural isomorphism in cohomology
\[
H^\bullet_{\mathrm{dR}}(M/H)\cong
H^\bullet\bigl(\Omega^\bullet(M,\mathcal F)^H,d\bigr).
\]
\end{thm}

\begin{proof}
Let $\pi_0:M\to X=M/H_0$ and $\rho:X\to X/K\cong M/H$ be the quotient maps,
so that $\pi_H=\rho\circ\pi_0$. Since $H_0$ is normal in $H$, the right
$H$-action on $M$ descends to a right $K$-action on $X$.

Applying Theorem~\ref{Hector} to the foliation $\mathcal F$, pullback by
$\pi_0$ gives an isomorphism of cochain complexes
\[
\pi_0^*:\Omega^\bullet(X)\xrightarrow{\;\cong\;}
\Omega^\bullet(M,\mathcal F).
\]
This isomorphism identifies $K$-invariant forms on $X$ with $H$-invariant
$\mathcal F$-basic forms on $M$. Indeed, if $h\in H$ and $\bar h\in K$
denotes its class, then $\pi_0\circ R_h=R_{\bar h}\circ\pi_0$. Hence, for
$\beta\in\Omega^\bullet(X)$, one has
$R_h^*(\pi_0^*\beta)=\pi_0^*(R_{\bar h}^*\beta)$. Since $\pi_0$ is a
subduction, $\pi_0^*$ is injective, and therefore $\pi_0^*\beta$ is
$H$-invariant if and only if $\beta$ is $K$-invariant. Thus
$\Omega^\bullet(M,\mathcal F)^H\cong \Omega^\bullet(X)^K$.

By Lemma~\ref{component-descent}, under either hypothesis (a) or (b),
pullback by $\rho$ induces an isomorphism
\[
\rho^*:\Omega^\bullet(X/K)\xrightarrow{\;\cong\;}\Omega^\bullet(X)^K.
\]
Combining this with the preceding isomorphism and using
$X/K\cong M/H$ and $\pi_H=\rho\circ\pi_0$, we obtain
\[
\pi_H^*:\Omega^\bullet(M/H)\xrightarrow{\;\cong\;}
\Omega^\bullet(M,\mathcal F)^H.
\]
Since pullback commutes with the exterior differential, this is an
isomorphism of cochain complexes. The cohomological statement follows.
\end{proof}

The subduction condition on the canonical map~\eqref{canonical-map} is
not automatic in general. The following example shows that, without this
condition, the conclusion of Theorem~\ref{M/H} can fail.

\begin{ex}\label{ex:subduction-fails}
Let $H=\mathbb R_{\mathrm{disc}}$ be the additive group of real numbers
endowed with the discrete Lie group structure, and let $M=\mathbb R$. Let
$H$ act freely and smoothly on $M$ on the right by translations,
$x\cdot h:=x+h$. Then $H_0=\{0\}$, so the foliation $\mathcal F$ by
$H_0$-orbits is the point foliation. Hence
$X:=M/H_0=M=\mathbb R$ and $K:=H/H_0=H=\mathbb R_{\mathrm{disc}}$.

The quotient $X/K$ is a single point, since the translation action of $H$ on
$\mathbb R$ is transitive. Therefore $X\times_{X/K}X=X\times X=\mathbb R^2$,
with its usual diffeology. The canonical map in Theorem~\ref{M/H} is
\[
\Delta:\mathbb R\times\mathbb R_{\mathrm{disc}}\longrightarrow
\mathbb R^2,\qquad
(x,h)\longmapsto (x,x+h).
\]
This map is not a subduction. Indeed, consider the smooth plot
$p:\mathbb R\to\mathbb R^2$ given by $p(t)=(0,t)$. If $\Delta$ were a
subduction, then, after shrinking to a connected open neighborhood $V$ of
$0$, the plot $p|_V$ would admit a lift through $\Delta$. Thus there would
exist smooth maps $x:V\to\mathbb R$ and
$h:V\to\mathbb R_{\mathrm{disc}}$ such that
$\Delta(x(t),h(t))=(0,t)$. The first component gives $x(t)=0$, and the
second gives $h(t)=t$. But a smooth map from a connected open interval to
the discrete space $\mathbb R_{\mathrm{disc}}$ is locally constant, whereas
$t\mapsto t$ is not. Hence $\Delta$ is not a subduction.

This example also shows why the subduction hypothesis is necessary, even for
a free smooth action. Since $M/H$ is a point, $\Omega^1(M/H)=0$. On the
other hand, since $\mathcal F$ is the point foliation, every ordinary form on
$\mathbb R$ is $\mathcal F$-basic. Moreover, $dx$ is $H$-invariant, since
$R_h^*dx=d(x+h)=dx$ for every $h\in H$. Hence
\[
0\neq dx\in\Omega^1(M,\mathcal F)^H.
\]
Therefore the pullback
$\pi_H^*:\Omega^\bullet(M/H)\to\Omega^\bullet(M,\mathcal F)^H$
is not surjective in degree $1$ if the subduction hypothesis is omitted.
\end{ex}

The preceding example shows that the subduction hypothesis is genuinely
needed in the non-second-countable case. The next example shows, in the
opposite direction, that second countability of the acting Lie group does
not itself imply the subduction condition, even for a smooth locally free
action. Thus the two hypotheses in Theorem~\ref{M/H} should be viewed as
distinct sufficient mechanisms rather than as variants of a single condition.

\begin{ex}\label{ex:second-countable-not-subduction}
Let $H=\mathbb Z$ act smoothly on the right on $M=\mathbb R$ by
$x\cdot n=2^n x$. The action is not free because the stabilizer of $0$ is
all of $\mathbb Z$, but it is locally free because every stabilizer is a
subgroup of the discrete group $\mathbb Z$. Here $H$ is second countable,
$H_0=\{0\}$, and the foliation $\mathcal F$ by $H_0$-orbits is the
foliation by points. Thus $X=M/H_0=\mathbb R$ and $K=H/H_0=\mathbb Z$.

The canonical map in this case is
\[
\Phi:X\times K\longrightarrow X\times_{X/K}X,\qquad
(x,n)\longmapsto(x,2^n x).
\]
We show that $\Phi$ is not a subduction.

Choose a nonnegative function $\varphi\in C_c^\infty(\mathbb R)$ supported
in $[0,1]$ and positive on $(0,1)$; for example, take
$\varphi(t)=e^{-1/(t(1-t))}$ for $0<t<1$ and $\varphi(t)=0$ otherwise.
For each $k\geq1$, set $I_k=(2^{-k},2^{-k}+2^{-2k})$ and define
$f_k(t)=e^{-k^2}\varphi\bigl(2^{2k}(t-2^{-k})\bigr)$. Then the intervals
$I_k$ are pairwise disjoint and accumulate only at $0$. Moreover, $f_k$ is
supported in $\overline I_k$, positive on $I_k$, and vanishes to infinite
order at the endpoints. For every $m\geq0$, there exists a constant
$C_m>0$, independent of $k$, such that
\[
\|f_k^{(m)}\|_\infty\leq C_m e^{-k^2}2^{2km},\qquad
\|(2^k f_k)^{(m)}\|_\infty\leq C_m e^{-k^2}2^{k+2km}.
\]
In particular, for every $m\geq0$, all derivatives of order $m$ of both
$f_k$ and $2^k f_k$ tend to $0$ as $k\to\infty$.

Define $\alpha,\beta:\mathbb R\to\mathbb R$ by setting
$\alpha(t)=f_k(t)$ and $\beta(t)=2^k f_k(t)$ for $t\in I_k$, and by setting
$\alpha(t)=\beta(t)=0$ outside $\bigcup_{k\geq1}I_k$. Since the intervals
$I_k$ are pairwise disjoint and accumulate only at $0$, the preceding
estimates imply that $\alpha$ and $\beta$ are smooth and vanish to infinite
order at $0$.

Set $p(t):=(\alpha(t),\beta(t))$. Then $p$ is a smooth plot of
$X\times_{X/K}X$ with the subset diffeology inherited from
$X\times X=\mathbb R^2$. Indeed, for each $t$, one has
$\beta(t)=2^n\alpha(t)$ for some $n\in\mathbb Z$.

Suppose, for contradiction, that $\Phi$ is a subduction. Then the plot $p$
admits a local lift through $\Phi$ on some connected open neighborhood $V$
of $0$. Such a lift has the form
$\widetilde p(t)=(\alpha(t),n(t))$, where $n:V\to\mathbb Z$ is smooth and
$\beta(t)=2^{n(t)}\alpha(t)$. Since $\mathbb Z$ has the discrete
diffeology, $n$ is locally constant, and hence constant on $V$; write
$n(t)=n_0$.

Since the intervals $I_k$ accumulate at $0$, the neighborhood $V$ contains
$I_k$ for all sufficiently large $k$. For $t\in I_k$, we have
$\alpha(t)>0$ and $\beta(t)=2^k\alpha(t)$. Comparing this with
$\beta(t)=2^{n_0}\alpha(t)$ gives $n_0=k$. This is impossible for two
distinct sufficiently large values of $k$. Hence $p$ admits no local lift
through $\Phi$ near $0$, and therefore $\Phi$ is not a subduction.
\end{ex}

We now derive the dense subgroup theorem of Clark--Ziegler~\cite{CZ26} as a
direct consequence of Theorem~\ref{M/H}.

\begin{thm}\label{thm:CZ}
Let $H$ be a dense subgroup of a Lie group $G$, endowed with its canonical
Lie group structure of Proposition~\ref{Bour}. Write $\mathfrak h$ and
$\mathfrak g$ for the Lie algebras of $H$ and $G$, respectively, and denote
by $\mathcal F$ the foliation of $G$ by the right $H_0$-orbits. Equip $G/H$
with the quotient diffeology $\mathcal D_H$, with quotient map
$\pi:G\to G/H$.

Then $\mathfrak h$ is an ideal of $\mathfrak g$. Moreover, the pullback
$\pi^*$, followed by evaluation at the identity and multiplication by
$(-1)^k$ in degree $k$, gives a canonical isomorphism of cochain complexes
\[
\Phi:(\Omega^\bullet(G/H,\mathcal D_H),d)
\xrightarrow{\;\cong\;}
(\wedge^\bullet(\mathfrak g/\mathfrak h)^*,d_{\mathrm{CE}}),
\]
where
\[
\Phi^k(\eta)(\bar Y_1,\dots,\bar Y_k)
:=
(-1)^k(\pi^*\eta)_e(Y_1,\dots,Y_k),
\qquad
\bar Y_i=Y_i+\mathfrak h.
\]
Consequently, $H^\bullet_{\mathrm{dR}}(G/H,\mathcal D_H)\cong H^\bullet(\mathfrak g/\mathfrak h)$.
\end{thm}

\begin{proof}
By \cite[Chapter 3, \S 9, no.~2, Proposition~5]{Bour89}, since $H$ is dense
in $G$, the Lie algebra $\mathfrak h$ is an ideal of $\mathfrak g$, and
$H_0$ is normal in $G$. Hence $\mathfrak g/\mathfrak h$ is a Lie algebra. Since $H_0$ is normal in
$G$, the right $G$-action on $G$ preserves $\mathcal F$, sending right
$H_0$-orbits to right $H_0$-orbits.

We apply Theorem~\ref{M/H} to the right $H$-action on $G$, which
is free. It remains to verify the subduction hypothesis. Set $X:=G/H_0$ and
$K:=H/H_0$. Let
\[
\Delta:X\times K\longrightarrow X\times_{X/K}X,\qquad
(x,k)\longmapsto (x,x\cdot k).
\]
Let $(a,b):U\to X\times_{X/K}X$ be a plot. Since $X=G/H_0$ carries the
quotient diffeology, locally on $U$ we may choose smooth lifts
$\widetilde a,\widetilde b:V\to G$ such that
$a|_V=\pi_{H_0}\circ\widetilde a$ and
$b|_V=\pi_{H_0}\circ\widetilde b$. Since $a$ and $b$ have the same image in
$X/K\cong G/H$, the division map $r:=\widetilde a^{-1}\widetilde b$ takes
values in $H$. By Proposition~\ref{Bour}, $r$ is smooth as a map $V\to H$.
Thus $q\circ r:V\to K$, where $q:H\to K$ is the quotient map, is a plot, and
$b|_V=a|_V\cdot(q\circ r)$. Hence $(a,q\circ r):V\to X\times K$ is a local
lift of $(a,b)$ through $\Delta$, so $\Delta$ is a subduction.

Theorem~\ref{M/H} gives an isomorphism of cochain complexes
\[
\pi^*:\Omega^\bullet(G/H,\mathcal D_H)
\xrightarrow{\;\cong\;}
\Omega^\bullet(G,\mathcal F)^{R(H)}.
\]
Since $H$ is dense in $G$, every right $H$-invariant smooth form on $G$ is
right $G$-invariant. Therefore
\[
\Omega^\bullet(G,\mathcal F)^{R(H)}
=
\Omega^\bullet(G,\mathcal F)^{R(G)}.
\]

We now identify $\Omega^\bullet(G,\mathcal F)^{R(G)}$ with
$\wedge^\bullet(\mathfrak g/\mathfrak h)^*$. If
$\omega\in\Omega^k(G,\mathcal F)^{R(G)}$, then $\omega_e$ vanishes whenever
one argument lies in $\mathfrak h=T_eH_0$, because $\omega$ is horizontal for
the right $H_0$-foliation. Hence $\omega_e$ descends to an element of
$\wedge^k(\mathfrak g/\mathfrak h)^*$.

Conversely, let $\alpha\in\wedge^k(\mathfrak g/\mathfrak h)^*$. For
$g\in G$, set $\rho_g:=d(R_{g^{-1}})_g:T_gG\to\mathfrak g$. Define a
right $G$-invariant $k$-form $\omega^\alpha$ on $G$ by
\[
\omega^\alpha_g(v_1,\dots,v_k)
:=
\alpha\bigl(\overline{\rho_g(v_1)},\dots,\overline{\rho_g(v_k)}\bigr).
\]
This form is basic for $\mathcal F$. Indeed, the tangent space to the right
$H_0$-orbit through $g$ is $d(L_g)_e\mathfrak h$, and for
$Z\in\mathfrak h$ one has
$\rho_g(d(L_g)_eZ)=\operatorname{Ad}_g Z\in\mathfrak h$, since
$\mathfrak h$ is an ideal. Thus $\omega^\alpha$ is horizontal. Its right
$G$-invariance implies invariance under the right $H_0$-action, so
$\omega^\alpha$ is basic. Therefore evaluation at the identity gives an
isomorphism of graded vector spaces
\[
\mathrm{ev}_e:\Omega^k(G,\mathcal F)^{R(G)}
\longrightarrow \wedge^k(\mathfrak g/\mathfrak h)^*.
\]

It remains to record the differential. Let $d_{\mathrm{CE}}$ denote the
Chevalley--Eilenberg differential on
$\wedge^\bullet(\mathfrak g/\mathfrak h)^*$; see, for example, \cite{K88}.
Let $Y^R$ denote the right-invariant vector field with value
$Y\in\mathfrak g$ at the identity. Since
$[Y_i^R,Y_j^R]=-[Y_i,Y_j]^R$ and the values of a right $G$-invariant form on
right-invariant vector fields are constant, for
$\omega\in\Omega^k(G,\mathcal F)^{R(G)}$ and
$\alpha=\mathrm{ev}_e(\omega)$ one has
$\mathrm{ev}_e(d\omega)=-d_{\mathrm{CE}}\alpha$. Hence multiplication by
$(-1)^k$ in degree $k$ turns evaluation at the identity into a cochain
isomorphism, $\omega\mapsto (-1)^k\mathrm{ev}_e(\omega)$.

Combining this cochain isomorphism with the isomorphism from
Theorem~\ref{M/H} gives the stated canonical isomorphism
\[
\Phi:(\Omega^\bullet(G/H,\mathcal D_H),d)
\xrightarrow{\;\cong\;}
(\wedge^\bullet(\mathfrak g/\mathfrak h)^*,d_{\mathrm{CE}}).
\]
Taking cohomology gives
\[
H^\bullet_{\mathrm{dR}}(G/H,\mathcal D_H)
\cong H^\bullet(\mathfrak g/\mathfrak h).
\]
This completes the proof.
\end{proof}

To finish the section, we apply Theorem~\ref{M/H} to a concrete nonclosed
situation not covered by the Clark--Ziegler dense subgroup theorem. This
shows that the foliation-theoretic viewpoint applies beyond the dense
subgroup setting.

\begin{ex}\label{ex:non-second-countable-nondense}
Let $G=SU(2)$, and let
\[
T=
\left\{
\begin{pmatrix}
z&0\\
0&\overline z
\end{pmatrix}
\,\middle|\,
z\in S^1
\right\}
\subset G
\]
be a maximal torus. Fix $\alpha\in\mathbb R\setminus\mathbb Q$, and choose
an uncountable proper $\mathbb Q$-vector subspace $V\subset\mathbb R$
containing $1$ and $\alpha$. Such a subspace exists by a standard
Hamel-basis argument. Set
\[
H=
\left\{
\begin{pmatrix}
e^{2\pi it}&0\\
0&e^{-2\pi it}
\end{pmatrix}
\,\middle|\,
t\in V
\right\}.
\]
Since $\mathbb Z\subset V$, we have $H\cong V/\mathbb Z$, so $H$ is
uncountable. Since $\alpha\in V$, the subgroup $H$ is dense in $T$, while
the properness of $V$ implies that $H\neq T$. Hence $\overline H=T$, so
$H$ is nonclosed in $G$ but is not dense in $G$.

The canonical Lie algebra of $H$ is zero. Indeed, if
$0\neq Z\in\mathfrak h\subset\mathfrak t$, then
$\exp(\mathbb RZ)=T\subset H$, contradicting $H\neq T$.
Thus the canonical Lie group structure on $H$ is discrete. Since $H$
is uncountable, it is not second countable. In particular, $H_0=\{e\}$,
and the right action of $H$ on $G$ is free.

The natural map $G/H\to G/T\cong S^2$ has fiber $T/H$ and is obtained from
the Hopf fibration $T\to SU(2)\to S^2$ by quotienting its fibers by $H$.
Thus this quotient arises from the nontrivial geometry of the Hopf
fibration rather than from a direct-product construction.

The verification of the subduction hypothesis in the proof of
Theorem~\ref{thm:CZ} does not use the density of $H$ in $G$ and applies
here verbatim. Hence condition~{\rm(b)} of Theorem~\ref{M/H} holds.

Since $H_0=\{e\}$, the corresponding
foliation is the point foliation, and hence
\[
\Omega^\bullet(G/H)\cong\Omega^\bullet(G)^H.
\]
Because $H$ is dense in $T$, every smooth differential form on $G$ that is
invariant under right translations by $H$ is invariant under right
translations by $T$. Therefore
$\Omega^\bullet(G)^H=\Omega^\bullet(G)^T$.

Since $T$ is compact and connected, averaging over $T$ shows that the
inclusion $\Omega^\bullet(G)^T\hookrightarrow\Omega^\bullet(G)$ is a
quasi-isomorphism. Consequently,
\[
H^\bullet_{\mathrm{dR}}(G/H)
\cong H^\bullet_{\mathrm{dR}}(SU(2))
\cong H^\bullet_{\mathrm{dR}}(S^3).
\]
Thus
\[
H^k_{\mathrm{dR}}(G/H)\cong
\begin{cases}
\mathbb R, & k=0,3,\\
0, & \text{otherwise}.
\end{cases}
\]
Equivalently, as a graded algebra,
$H^\bullet_{\mathrm{dR}}(G/H)\cong\Lambda(u)$ with $\deg u=3$.
\end{ex}

\end{document}